\documentclass[11pt,reqno]{amsart}

\usepackage{amsmath}
\usepackage{amssymb}
\usepackage{amsthm}
\usepackage{comment}

\newtheorem{lemma}{Lemma}
\newtheorem{theorem}{Theorem}
\newtheorem*{theorem1}{Theorem}

\newcommand{\grad}{\nabla}
\newcommand{\Bop}{\mathcal{B}}

\newcommand{\e}{\varepsilon}
\DeclareMathOperator*{\argmin}{arg\,min}
\title{Variational properties of the kinetic solutions of scalar conservation laws}

\author{Misha Perepelitsa}
\address{Misha Perepelitsa, Department of Mathematics, University of Houston, 651 PGH, Houston,
Texas 77204-3008.} 

\email{\tt misha@math.uh.edu}       
        
\date{\today}
\begin{document}

\begin{abstract}
We discuss properties of kinetic solutions of scalar conservation laws in the variational approach developed by Panov\cite{Panov1, Panov2} and also Brenier\cite{Brenier1}.
Our main result shows that such solutions can be considered as curves in a suitable Hilbert space with tangents that are unique minimizers of an interaction functional.  
\end{abstract}
\maketitle

\begin{section}{Introduction}
We consider a scalar conservation law
\begin{equation}
\label{SCL}
u_t{}+{}\grad_x\cdot f(u){}={}0,\quad (t,x){}\in\mathbb{R}_+\times\mathbb{R}^n,
\end{equation}
with the flux $f{}:{}\mathbb{R}{}\to{}\mathbb{R}^n.$ A function $u(t,x)$ is called an entropy weak solution if for any convex entropy/entropy flux pairs $(\eta(u(t,x)),q(u(t,x))$
of the flux function $f,$
\begin{equation}
\eta_t{}+{}\grad_x\cdot q{}\leq{}0,\quad \mathcal{D}'(\mathbb{R}^{n+1}_+).
\end{equation}
The existence of unique entropy solutions for \eqref{SCL} with the initial data $u(t=0){}={}u_0{}\in{}L^\infty(\mathbb{R}^n),$ as well as their stability,  was obtained in Kruzhkov\cite{Kruzhkov}. The solution can be described using the kinetic formulation, as proved by  Lions-Perthame-Tadmor\cite{LPT}. In this approach, $u(t,x)$ is a weak entropy solution iff  the kinetic density function
\begin{equation}
\label{KD}
Y(t,x,v){}={}\left\{ 
\begin{array}{rc}
1 & v{}\geq{}u(t,x),\\
0 & v{}<{}u(t,x),
\end{array}
\right.
\end{equation}
verifies the transport equation
\begin{equation}
\label{kinetic_1}
Y_t{}+{}f_v(v)\cdot\grad_x Y{}={}-\partial_v m,\quad \mathcal{D}'(\mathbb{R}^{n+2}_+),
\end{equation}
for some measure $m\in \mathcal{M}_+(\mathbb{R}^{n+2}_+).$ To be precise, in \cite{LPT}, the kinetic denisty $\chi(t,x,v){}={}H(v){}-{}Y(t,x,v),$ where $H=1,\,v\geq0,$ $H{}={}0,\,v<0,$ was used, but the result can be expressed through $Y(t,x,v)$ as well.

Condition  \eqref{kinetic_1} can be equivalently expressed via a variational form:  for any regular test functions $\tilde{Y}(x,v),$ nondecreasing in $v,$ it holds
\begin{equation}
\label{kinetic_2}
\iint (\tilde{Y}-Y)(Y_t{}+{}f_v\cdot\grad_xY)\,dxdv{}\geq{}0,\quad \mathcal{D}'(\mathbb{R}).
\end{equation}
Indeed, the equivalence holds because,
\[
\iint Y(Y_t{}+{}f_v\cdot\grad_xY)\,dxdv{}={}\partial_t\int|u(t,x)|\,dx{}={}0,
\]
if, for example, $u_0\in L^1(\mathbb{R}^n ),$ or periodic. 
An interesting kinetic formulation for weak solutions were obtained in Panov\cite{Panov1, Panov2}, by allowing generic nondecreasing in $v$ functions $Y(t,x,v)$ in \eqref{kinetic_2}, rather than functions with the range in $\{-1,0,1\}.$

Panov\cite{Panov1} introduced a class of strong measure-valued solutions on \eqref{SCL} as
the set of parametrized probability measures $\nu_{t,x}\in \mathcal{M}_+(\mathbb{R}),$ for which the level sets of their distribution functions $Y(t,x,v):=\nu_{t,x}((-\infty,v])$ are the graphs of weak entropy solutions of \eqref{SCL}, i.e. 
\begin{equation}
\label{kinetic_3}
u(t,x,\lambda){}={}\sup\{ v\,:\, Y(t,x,v)\leq \lambda\},
\end{equation}
is an entropy weak solution of \eqref{SCL} for any $\lambda\in (0,1).$ It was shown
that for each initial data $\nu_{0,x}$  there is a unique, global strong measure-valued solution. The weak entropy solutions of \eqref{SCL} are naturally contained in this approach as measures 
\begin{equation}
\label{kinetic:mv}
\nu_{t,x}{}={}\delta_{u(t,x)}.
\end{equation}

The set strong measure-valued solutions is the subclass of the entropy measure-valued solutions were introduced by Tartar\cite{Tartar} in his compensated compactness method. Such solutions were further studied by DiPerna\cite{DiPerna}, who showed that
the weak entropy solutions are unique in the class of measure-valued solutions and by Schochet\cite{Schochet}, who  showed that entropy measure-valued solutions with a prescribed initial data are not unique. 

An equivalent formulation of a strong measure-valued solution was given in Panov\cite{Panov2} where it was shown that $Y(t,x,v)$ is the distribution of a strong measure valued solution iff it verifies \eqref{kinetic_2}.
Note that the formulation \eqref{kinetic_2} in addition to \eqref{kinetic_1} prescribes a non-trivial, non-linear constraint:
\begin{equation}
\label{kinetic_5:2}
\iint Y^2(t,x,v)\,dxdv{}={}const.,\quad \mathrm{a.e.}{} t>0.
\end{equation}
This  result was later re-discovered by Brenier\cite{Brenier1} in the following form. He proved that  $Y$ is the solution of \eqref{kinetic_2} iff $u(t,x,\lambda)$ from \eqref{kinetic_3} is an entropy weak solution of \eqref{SCL} for any $\lambda$ in the interior of  the range of $Y.$  In that paper, the variational formulation
\eqref{kinetic_2} is expressed as a differential inclusion on a Hilbert space of $L^2$ integrable in $(x,v)$ functions:
\begin{equation}
\label{kinetic_4}
Y_t{}\in {} - f_v\cdot\grad_x Y{}-{}\partial K(Y),
\end{equation}
where $\partial K(Y)$ is the subdifferential to the indicator function of a convex, closed cone $K$ consisting of all non-decreasing in $v$ functions, see section \ref{examples} for details.
Operator appearing on the right of \eqref{kinetic_4} is monotone, providing the uniqueness of solutions  and if in addition it is maximal then the existence follows from the classical results, for example Brezis\cite{Brezis}. 

Additionally to the existence/uniqueness/stability of solution Brenier\cite{Brenier1} proves the {\it regularity} of solutions of \eqref{kinetic_4}: if the initial data $Y_{0}(x,v)$ is differentiable, $\grad_xY_0\in L^2_{x,v},$ then
\[
\partial_tY,\,\grad_xY{}\in{} L^\infty((0,+\infty);L^2_{x,v}).
\]
It can also be shown that $\partial_vY{}\in {} L^\infty((0,+\infty);L^2_{x,v}),$ if in addition $\partial_v Y_0\in L^2_{x,v}.$    

The results \cite{Panov1, Panov2, Brenier1} show a remarkable fact that {\it all} weak entropy solutions of \eqref{SCL} can be obtained through \eqref{kinetic_3} (or \eqref{kinetic:mv})  from a globally {\it stable} and {\it regular} (if $\partial_x Y_0\in L^2_{x,v}$)
kinetic densities $Y(t,x,v)$ -- solutions of \eqref{kinetic_4} (or \eqref{kinetic_2}). 

In this paper we further investigate the properties of solution of \eqref{kinetic_4}. Such solutions can be considered as curves $Y(t)$ with values in the admissible cone $K$ for which the tangent $\partial_t Y(t)$ belongs to the tangent cone $T_K(Y(t)){}={}
\mathrm{closure} \;\{ h(\tilde{Y}-Y(t))\;:\; h>0,\,\tilde{Y}\in K\}.$

 Our main result shows that $Y(t)$ is the solution of \eqref{kinetic_4} iff $\partial_t Y(t)$ minimizes the functional $\|V+f_v\cdot\grad_x Y(t)\|_{L^2_{x,v}}$ over all directions $V\in T_K(Y(t)).$ This functional can considered as an ``instantaneous'' interaction functional. Since set $K$ restrict solutions only in $v$ direction this functional is local in $x,$ i.e.,
\[
\min_{T_K(Y(t))} \|V+f_v\cdot\grad_x Y(t)\|_{L^2_{x,v}}{}={}
\| \min_{T_K(Y(t,x,\cdot))}\|V(\cdot){}+{}f_v(\cdot)\cdot\grad_x Y(t,x,\cdot)\|_{L^2_v}\|_{L^2_x}
\]
 Zero minimal value is attained on solutions $Y(t)$ that are simply transported, i.e. $Y_t{}+{}f_v\cdot\grad_x Y{}={}0,$ while remaining in $K.$ At the level of weak entropy solutions of \eqref{SCL} such $Y(t)$ corresponds to classical solutions. For shock waves it is proportional to the shock strength, see section \ref{examples}.

After the prove of this result, which based on the fact that solutions of $\partial_tY{}\in{}-\mathcal{A}(Y(t)),$ with maximal monotone operator $\mathcal{A}$ are {\it slow} solutions, i.e., the solutions for which $\|\mathcal{A}(Y(t))\|$ is minimal, we show that there are travelling wave solution to \eqref{kinetic_4} that correspond to the shock waves of \eqref{SCL}. Such travelling waves move with the actual shock speed $\sigma{}={}\Delta f/\Delta u.$ The shock speed appears in solving a minimization problem
\[
\min_{V\in T_K(Y(t))}\|V{}+{}f_v\cdot\grad_x Y(t)\|_{L^2_{x,v}}.
\]

The shock profiles of this type are obtained by smoothing in $x$ direction the kinetic density \eqref{KD} of the shock wave $u(t,x).$  This however is rather exceptional case. In the last part of this paper we give an example that shows that generically if $u(t,x)$ is a weak entropy solution  that contain interacting waves and if $Y(t,x,v)$ -- its kinetic density, then $Y_\e{}={}Y(t,x,v)*\omega_\e(x),$  is not a solutions of \eqref{kinetic_4}. This happens because the constraint \eqref{kinetic_5:2} is non-linear in $Y$ and  does not commute with averaging.

\end{section}

\begin{section}{General theory}

Let $\mathbb{H}$ be the space of $2L$--periodic in $x,$ functions $Y(x,u)$ of $(x,u)\in\mathbb{R}^n\times[0,1],$ with the norm
\[
\|Y\|^2{}={}\langle Y,Y\rangle {}={}\int_{\Pi} \int_0^1 Y^2(x,u)\,dudx,\quad \Pi{}={}[-L,L]^n.
\]
Let $K\subset \mathbb{H}$ b a set of $Y$'s non-decreasing in $u.$ $K$ is a closed cone and for any $Y\in K,$ we denote 
\begin{equation}
\label{tangentcone}
T_K(Y){}={}\mbox{$\mathbb{H}$-- closure of }\{h(\tilde{Y}-Y)\,:\,h\geq0,\,\tilde{Y}\in K\},
\end{equation}
the tangent cone to $K$ at $Y$ and the normal cone:
\begin{equation}
\label{normalcone}
\partial K(Y){}={}\{Z\in\mathbb{H}\,:\,\langle Z, \tilde{Y}-Y\rangle{}\leq{}0,\,\forall \tilde{Y}\in K\}.
\end{equation}
Also by $N{}={}\{Z\in\mathbb{H}\,:\,\langle Z, \tilde{Y}\rangle{}\leq{}0,\,\forall \tilde{Y}\in K\}$ we denote the polar cone to $K.$ 
We consider the Cauchy problem
\begin{equation}
\label{prob:main}
\partial_t Y{}+{}f_u\cdot\grad_x Y{}\in{} -\partial K(Y),\quad Y(t=0){}={}Y_0.
\end{equation}
The flux function  $f\in Lip([0,1])^n$ -- Lipschitz continuous on $[0,1].$ 

It was shown in Brenier\cite{Brenier1} that for any $Y_0\in K,$ there is a unique solutions of \eqref{prob:main}
 $Y\in C([0,+\infty);\mathbb{H}),$ for any $t\geq0, $ $Y(t)\in K$, and if $\grad_x Y_0\in\mathbb{H},$ then also
\[
\partial_t Y,\,\grad_x Y{}\in{} L^\infty((0,+\infty);\mathbb{H}).
\]

Our main result contained in the following theorem.
\begin{theorem}
Let $Y_0\in K$ and $\grad_x Y_0\in\mathbb{H}.$ For the solution $Y\in C([0,+\infty);\mathbb{H})$ of \eqref{prob:main},  a.e. $t>0,$
\begin{equation}
\label{slow1}
\|\partial_t Y(t)\|{}={}\min_{Z\in \partial K(Y(t))}||Z{}+{} f_v\cdot\grad_x Y(t)\|,
\end{equation}
and 
\begin{equation}
\label{min}
\|\partial_tY(t){}+{}f_v\cdot \grad_x Y(t)\|{}={}\min_{V\in T_K(Y(t))}\|V{}+{}f_v\cdot \grad_x Y(t)\|.
\end{equation}
Conversely, each of the conditions \eqref{slow1}, \eqref{min} defines a unique solution of the problem \eqref{prob:main}. 
\end{theorem}
\begin{proof}
Let $a(v){}={}f_v(v)/|f_v(v)|,$ $v\in[0,1].$ If $|f_v|>0,$ we define $\mathcal{B}(Y){}={}|f_v(v)|\partial_{a(v)}Y,$ where $\partial_a Y$ is the derivative of $Y$ in the direction $a.$ The domain of this operator, $D(\mathcal{B})$ consists of all functions $ Y{}\in{}\mathbb{H}$
that are weakly differentiable in the direction $a(v)$ for a.e. $v\in[0,1]$ and such that
$\partial_a Y{}\in{} \mathbb{H}.$
\begin{lemma}
\label{B_MMO}
Let 
\begin{equation}
\label{non_degen}
c_0{}={}\mathrm{ess} \inf_{[0,1]}|f_v(v)|{}>{}0.
\end{equation}
 Then, $\mathcal{B}$ is a maximal monotone operator.
\end{lemma}
\begin{proof}
Monotonicity of $\mathcal{B}$ follows directly from the definition of $\mathcal{B}$ and periodicity of $Y$ in $x.$ To show maximality, let $W{}\in{}\mathbb{H},$ and $\lambda>0$ consider a problem:
\[
Y{}+{}\lambda |f_v|\partial_aY{}={}W.
\]
For a.e. $v\in[0,1],$ $W_0(v,\cdot)\in L^2(\Pi)$ and and the equation can be integrated along the characteristic to obtain a periodic solutions $Y(v,\cdot).$ The inclusion $Y\in D(\mathcal{B})$ follows from the a priori estimates
\[
\| Y\|{}\leq{}\|W\|,\quad c_0\|\partial_a Y\|{}\leq{}(\sqrt{\lambda})^{-1}\|W\|.
\]
\end{proof}

\begin{lemma}
\label{MMO}
Under the condition on $f$ from the previous lemma, $\mathcal{B}+\partial K$ is a maximal monotone operator on $\mathbb{H}.$
\end{lemma}

\begin{proof}
Consider now a proper, l.s.c., convex function 
\[
K(Y){}={}\left\{
\begin{array}{ll}
0, & x\in K,\\
+\infty, & x\not\in K.
\end{array}\right.
\]
The subdifferential $\partial K$ is a maximal monotone operator. The Yosida approximation of $\partial K(Y),$ equals $\nabla K_\lambda(Y),$ where
\[
K_\lambda(Y){}={}\inf_{\tilde{Y}\in\mathbb{H}}\left( K(\tilde{Y}){}+{}\frac{1}{2\lambda}\|\tilde{Y}-Y\|^2\right),
\]
see Theorem 4, p. 162 of Aubin-Cellina\cite{Aubin1}. It follows that $K_\lambda(Y){}={}\inf_{\tilde{Y}\in K}\frac{1}{2\lambda}\|\tilde{Y}-Y\|^2,$ and
\[
\nabla K_\lambda(Y){}={}\frac{1}{\lambda}\pi_N(Y),
\]
where $\pi_N(Y)$ is a projection of $Y$ onto $N$ -- the polar cone to $K.$ Operator $\frac{1}{\lambda}\pi_N(\cdot)$ is a Lipschitz continuous operator from $\mathbb{H}$ to $\mathbb{H},$  with the Lipschitz constant $\frac{1}{\lambda}$ and is monotone. 

From lemma \ref{B_MMO} and lemma 2.4 of Brezis\cite{Brezis} it follows that $\mathcal{B}{}+{}\frac{1}{\lambda}\pi_N(\cdot)$ is a maximal monotone operator.
So, for any $W\in\mathbb{H}$ and $\alpha>0,$ there is a solution $Y^\lambda$ of 
\begin{equation}
\label{sur}
Y{}+{}\alpha\left(\Bop(Y)+\grad K^\lambda(Y)\right){}={}W.
\end{equation}
Using monotonicity we get 
\[
\|Y^\lambda\|{}\leq{}\|W\|.
\]
Also, since 
\[
\langle \grad K^\lambda(Y^\lambda),\partial_a Y^\lambda\rangle{}={}
\langle \frac{1}{\lambda}\pi_N(Y^\lambda), \partial_a Y^\lambda\rangle{}={}0,
\]
we obtain
\[
\alpha\langle |f_v|\partial_a Y^\lambda,\partial_a Y^\lambda\rangle{}={}-\langle W, \partial_a Y^\lambda\rangle,
\]
and consequently,
\begin{eqnarray}
\label{uniform_1}
&&\|\partial_a Y^\lambda\|{}\leq{} C(\alpha, c_0)\|W\|,\\
\label{uniform_2}
&&\|\grad K^\lambda(Y^\lambda)\|{}\leq{}C(\alpha, \mathrm{ess}\sup_v|f_v(v)|, c_0)\|W\|.
\end{eqnarray}
We want to pass to the limit $\lambda\to0$ in the equation \eqref{sur}. We have shown that all terms in that equation are weakly compact in $\mathbb{H}.$ It remains to show that the sequence $Y^\lambda$ is strongly compact and use the strong-weak closeness of the maximal monotone operator $\partial K.$  Using the equation \eqref{sur} we compute
\begin{align*}
\|Y^\lambda- Y^\mu\|{}\leq{} \alpha (\lambda\langle \grad K^\lambda(Y^\lambda),\grad K^\mu(Y^\mu)\rangle{} +{}\mu\langle \grad K^\lambda(Y^\lambda),\grad K^\mu(Y^\mu)\rangle
\\
{}-{} \lambda\|\grad K^\lambda(Y^\lambda)\|^2{}-{}\mu \|\grad K^\mu(Y^\mu)\|^2.
\end{align*}
This estimate, due to \eqref{uniform_2}, implies that $Y^\lambda$ is Cauchy and converges to some $Y\in\mathbb{H}.$ Moreover $(I+\lambda\partial K)^{-1}(Y^\lambda){}={}\pi_K(Y){}={}Y^\lambda{}-{}\lambda \grad K^\lambda(Y^\lambda)$ converges to $Y.$ Since $\grad K^\lambda(Y^\lambda)\in \partial K((I+\lambda\partial K)^{-1}(Y^\lambda)),$ and $\partial K$ is strongly-weakly closed, it follows that
$\grad K^\lambda(Y^\lambda){}\to{}\partial K(Y),$ and $Y$ is the solution of 
\[
Y{}+{}\alpha(\Bop(Y){}+{}\partial K(Y)){}={}W,
\]  
proving by this the maximality of $\Bop{}+{}\partial K.$
\end{proof}

Consider a Cauchy problem
\begin{equation}
\label{CP1}
\partial_tY{}\in{}-\Bop(Y){}-{}\partial K(Y),\;Y(t=0){}={}Y_0.
\end{equation}
Under the non-degeneracy condition \eqref{non_degen}, $\Bop +\partial K$ is maximal monotone and the  problem \eqref{CP1} has a unique solution with the properties listed in the next theorem, see  theorem 1, p.142 of \cite{Aubin1}.
\begin{theorem1}
\label{exist2}
Let $Y_0{}\in{}D(\mathcal{B})\cap K.$ There is a unique solution $Y(t)$ of \eqref{CP1} for $t\in[0,+\infty),$ with the following properties: $Y(t)\in D(\mathcal{B})\cap K,$
\begin{equation*}
Y{}\in{} C([0,T];\mathbb{H}),\,\forall T>0,\, \partial_t Y,\,\partial_a Y{}\in{} L^\infty(0,+\infty;\mathbb{H}).
\end{equation*}
Moreover, 
\begin{enumerate}

%\item $\|\partial_tY(t)\|$ is non-increasing in $t.$
%\item  If $Y^1(t),\,Y^2(t)$ are two solutions of the differential equation in \eqref{CP1} with initial data $Y^1_0,$ and $Y^2_0,$ respectively, then
%for any $t>0,$
%\begin{equation}
%\label{non-expansive1}
%\|Y^1(t){}-{}Y^2(t)\|{}\leq{}\|Y^1_0{}-{}Y^2_0\|.
%\end{equation}

\item If $\grad_x Y_0\in\mathbb{H},$ then for any $t>0,$ $\|\grad_x Y(t)\|{}\leq{}\|\partial_xY_0\|.$

\item $\partial_tY(\cdot)$ is continuous from the right on $[0,+\infty)$ and  $\|\partial_t Y(t)\|{}\leq{}\mathrm{ess}\sup_{u}|f_v(v)|\|\partial_x Y_0\|.$

\item For any $t>0,$
\begin{equation}
\label{slow}
\partial_tY(t){}={}\argmin_{\tilde{Y}\in -|f_v|\partial_aY(t){}-{}\partial K(Y(t))}\|\tilde{Y}\|.
\end{equation}
\end{enumerate}
\end{theorem1}

 Let $f\in Lip([0,1])^n$ and $f_\e$ be sequence of Lipschitz continuous vector functions such that: (i)  $f_\e\to f$ in $C([0,1])^n;$ (ii) $f_{\e,v}{}\to{}f_v,$ a.e. $u\in (0,1);$ (iii) $\| f_{\e,v}\|_{L^\infty((0,1))^n}$ -- uniformly bounded; (iv) for all $\e\in(0,\e_0),$ $\inf_{[0,1]}|f_{\e,v}|>0.$ For each $f_\e$ and $Y_0\in D(\mathcal{B})\cap K,$ there is a solution $Y_\e$ that solves \eqref{CP1} with $f_\e$ and verifies the conclusions of the cited above theorem.  It follows from the same theorem and assumptions on $f_\e$ that norms $\|Y_\e(t)\|,$ $\|\partial_t Y_\e(t)\|,$ $\|\grad_x Y_\e(t)\|$ are uniformly bounded in $(t,\e)\in [0,+\infty)\times (0,\e_0).$ Moreover, by monotonicity we obtain:
\begin{equation}
\|Y_{\e_1}(t){}-{}Y_{\e_2}(t)\|{}\leq{} t(\mathrm{ess}\sup_{v}|f_{\e_1,v}-f_{\e_2,v}|) \|\partial_x Y_0\|,\quad \forall \e_1,\e_2\in (0,\e_0),
\end{equation}
i.e. $Y_\e{}$ is compact in $C([0,+\infty); \mathbb{H}).$ With this and using the fact that $\partial K,$ as a maximal monotone operator is strongly-weakly closed, we obtain $Y{}={}\lim Y_\e$ -- the solution of \eqref{prob:main} with $\partial_t Y, \grad_x Y{}\in{}L^\infty((0,+\infty);\mathbb{H}).$

Let us prove property \eqref{slow1}. Consider operator $\mathcal{A}{}={}\Bop{}+{}\partial K.$ It is monotone and has a maximal extension  $\tilde{\mathcal{A}}.$ Thus $Y(t,x,u)$ is also a solution of $\partial_t Y{}\in{}-\tilde{\mathcal{A}}(Y),$ and by \eqref{slow},
$\|\partial_t Y(t)\|{}={}\min_{Z\in \tilde{\mathcal{A}}(Y(t))} \|Z\|,$ for a.e. $t.$ Since $\partial_t Y\in \mathcal{A}(Y(t)),$ and \eqref{slow1} follows.

Now we can prove \eqref{min}. For the tangent cone $T_K(Y),$ defined in the beginning of this section,  we have that $V\in T_K(Y)$ iff $\forall Z\in \partial K(Y),$ $\langle V,Z\rangle{}\leq 0,$  i.e. $T_K(Y)$ is the polar cone to a convex closed cone $\partial K(Y).$
Property \eqref{slow1} states that $\partial_t Y$ equals $(I-\pi_{\partial K})(-\Bop(Y)),$ where $\pi_{\partial K}$ is the projector onto $\partial K(Y).$ This  can be stated equivalently, that $\partial_t Y$ is the projection of $-f_u\cdot \grad_x Y$ onto $T_K(Y),$ or
\[
\|\partial_t Y(t){}+{}\Bop(Y(t))\|{}={}\min_{V\in T_K(Y(t))}\|V{}+{}\Bop(Y(t))\|,
\] 
for a.e. $t.$
\end{proof}

\end{section}

\begin{section}{Examples}
\label{examples}
Consider a scalar conservation law \eqref{SCL} in one dimension with a convex flux function $f(u).$ We prescribe the initial data
\[
u_0(x){}={}\left\{
\begin{array}{cl}
u^+, & x\in [-L,0]\cup[L/2,L],\\
u^-, & x\in (0,L/2),
\end{array}
 \right.
\]
with $u^+>u^-.$ The weak entropy solution $u(t,x)$ of \eqref{SCL} consists (for small times) of a shock wave propagating from $x=0$ with the speed 
\begin{equation}
\label{shock_speed}
\sigma{}={}(f(u^+)-f(u^-))/(u^+-u^-)
\end{equation}
 and a rarefaction wave centred at $x={}L/2.$ The solution has this structure until the moment the shock wave collides with the r-wave. Let us choose a small $\e>0$ and
consider the kinetic formulation for this problem. We define
\begin{equation}
\label{weakY}
\tilde{Y}(t,x,v){}={}\left\{\begin{array}{ll}
0, & v<u(t,x),\\
1,& v\geq u(t,x).
\end{array}\right.
\end{equation}
$Y(t,x,v)$ is the solution of the variational problem \eqref{kinetic_2}, but it is not differentiable in $x,$ and so we can not test it in the interaction functional \eqref{min}. We will approximate 
$Y(t,x,v)$ by 
\[
Y_\e(0,x,v){}={}Y(0,x,v)*\omega_\e(x),
\]
where $\omega_{\e}(x)$ is the standard (supported on $[x-\e,x+\e]$, non-negative, unit mass) smoothing kernel.

It can be verified that  $Y_\e(t,x,v)$ for all small $t,$ for which the shock wave and the r-wave in $u(t,x)$ are separated by the distance larder  than $4\e,$ is the solution of \eqref{kinetic_4} (or \eqref{kinetic_2}).  Indeed, $Y_\e$ verifies \eqref{kinetic_1} because a convolution in $x$ with a non-negative kernel doesn't change the structure of that equation, and, moreover it can be checked by a computation that the conservation property \eqref{kinetic_5:2} holds as well. The structure of $Y_\e$ is simple; it consists of a smoothed shock wave : for $x\in(\sigma t{}-{}2\e,\sigma t{}+{}2\e),$ it equals
\[
Y_\e(t,x,v){}={}\left\{
\begin{array}{ll}
1 & v>u^+,\\
\int_{-x}^{x+2\e}\omega_\e(y)\,dy & v\in[u^-,u^+],\\
0 & v<u^-,
\end{array}\right.
\] 
and the part that corresponds to the regularization of the rarefaction wave.
Next we would like to find the value of the interaction functional.
Let us fix time  $t=0$ and for $x\in [-L,L],$ consider
\begin{equation}
\label{Min_2}
\min_{V\in T_K(Y_\e(0,x,\cdot))}\|V{}+{}f_v(v)\partial_x Y_\e(0,x,v)\|_{L^2((-1,1))}^2.
\end{equation}
In the next lemma we show that the minimal value of  is zero when $x$ is in the range of the r-wave and it is proportional to the shock strength  $|u^+-u^-|$ for $x$ in the range of the shock discontinuity. For the $x$'s in the later case, $\partial_t Y_\e{}+{}\sigma\partial_x Y_\e{}={}0,$ where $\sigma$ from \eqref{shock_speed}.
\begin{lemma}
Let $\e<L/8.$ The minimizer $V_0$ of \eqref{Min_2} equals
\begin{equation}
\label{V_MIN}
V_0{}={}\partial_t Y_\e(0,x,v) {}={}\left\{ 
\begin{array}{ll}
-\sigma \partial_x Y_\e(0,x,v) & x\in (-2\e,2\e),\,v\in[0,1], \\
-f_v(v)\partial_x Y_\e(0,x,v) & x\in (-2\e+L/2,L/2+2\e), v\in[0,1],\\
0& \mathrm{otherwise,}
\end{array}
\right.
\end{equation}
with $\sigma$ from \eqref{shock_speed}.
The minimal value is proportional to the strength of the shock wave $|u^+-u^-|,$ for $ x\in (-4\e,4\e),$ and is $0$ for other values of $x.$
\end{lemma}
\begin{proof} Assume that $\sigma>0.$ The other case it treated similarly. The approximate initial datum $Y_\e(0,x,v),$ for $x\in(-2\e,2\e),$ equals
\[
Y_\e(0,x,v){}={}\left\{
\begin{array}{ll}
1 & v>u^+,\\
\int_{-x}^{x+2\e}\omega_\e(y)\,dy & v\in[u^-,u^+],\\
0 & v<u^-,
\end{array}\right.
\] 
and for $x\in (-2\e+L/2,L/2+2\e),$
\[
Y_\e(0,x,v){}={}\left\{
\begin{array}{ll}
1 & v>u^+,\\
\int^{L/2-x}_{L/2-2\e}\omega_\e(y)\,dy & v\in[u^-,u^+],\\
0 & v<u^-.
\end{array}\right.
\] 
$\partial_x Y_\e$ in the corresponding ranges equals
\[
\partial_x Y_\e(0,x,v){}={}\left\{
\begin{array}{ll}
0 & v>u^+,\\
\omega_\e(-x)& v\in[u^-,u^+],\\
0 & v<u^-,
\end{array}\right.
\]
for $x\in(-2\e,2\e),$  
and 
\[
\partial_x Y_\e(0,x,v){}={}\left\{
\begin{array}{ll}
0 & v>u^+,\\
-\omega_\e(L/2-x) & v\in[u^-,u^+],\\
0 & v<u^-,
\end{array}\right.
\] 
for  $x\in (-2\e+L/2,L/2+2\e).$
Then, with 
\[
V\in T_K(Y_\e(0,x,\cdot)){}={}\mbox{$L^2((0,1))$--closure of }\;{}  \{ h(\tilde{Y}(v)-Y_\e(0,x,v))\;:\;h\geq0,\,\partial_v \tilde{Y}\geq0\},
\]
\begin{eqnarray}
\| V{}+{}f_v(v)\partial_xY_\e(0,x,v)\|^2 &{}={}&
\int_0^{u^-}|V(v)|^2\,dv{}+{}\int_{u^+}^1|V(v)-1|^2\,dv \notag \\
&&{}+{}\int_{u^-}^{u^+}|V(v){}+{}f_v(v)\partial_x Y_\e(0,x,v)|^2\,dv.
\end{eqnarray}
Notice that, due to the fact that $Y_\e(0,x,v)$ takes only three values, all functions $V(v),$
such that  $V{}={}0,$ for $v\in[0,u^-),$ $V(v){}={}1,$ for $v\in(u^+,1],$ and $V(v)$ is non-decreasing on $[u^-,u^+],$ belong to $T_K(Y_\e(0,x,\cdot)).$ Thus,
\begin{multline}
\min_{V\in T_K(Y_\e(0,x,\cdot))} \| V{}+{}f_v(v)\partial_xY_\e(0,x,v)\|^2  \notag \\
\label{MIN3}
{}={}
\min_{V'(v)\geq0,\,v\in[u^-,u^+]} \int_{u^-}^{u^+}|V(v){}+{}f_v(v)\partial_x Y_\e(0,x,v)|^2\,dv.
\end{multline}
For  $x\in (-2\e+L/2,L/2+2\e)$ we can take $V{}={}\omega_\e(L/2-x)f_v(v),$ for $v\in[u^-,u^+].$ Such $V$ gives zero value of the functional. For the shock discontinuity range $x\in(-2\e,2\e),$ because $f_v(v)\partial_x Y_\e(0,x,v)$ is non-decreasing in $v,$
minimum  will be achieved on constant functions $V(v){}={}c:$
\begin{multline}
\min_{V'(v)\geq0,\,v\in[u^-,u^+]} \int_{u^-}^{u^+}|V(v){}+{}f_v(v)\partial_x Y_\e(0,x,v)|^2\,dv \notag \\
{}={}\min_{c} \int_{u^-}^{u^+}|c{}+{}f_v(v)\partial_x Y_\e(0,x,v)|^2\,dv.
\end{multline}
The later functional is minimized for $c{}={}-\sigma\partial_x Y_\e(0,x,v).$ This establishes \eqref{V_MIN}. It is easily verified that with such minimizer the value of \eqref{Min_2} is proportional to $|u^+-u^-|.$

\end{proof}

Next we will show that the regularizations $Y_\e{}={}Y*\sigma_\e(x)$ of the kinetic density $Y$ of a weak entropy solution $u(t,x)$ are not, in general, solutions of the variational problems \eqref{kinetic_4} (or \eqref{kinetic_2}).
For that we consider  a conservation law:
\begin{equation}
\label{scl1}
u_t{}+{}((u-\frac{1}{2})^2)_x{}={}0,
\end{equation}
with $2L$ periodic data 
\[
u_0(x){}={}\left\{\begin{array}{ll}
1, & x\in[-L,0],\\
0, & x\in(0,1),\\
1, & x\in[1,L].
\end{array}
\right.
\]
The corresponding entropy solution $u(t,x)$ consists of a stationary shock wave at $x=0,$ and a r-wave centred at $x=1$  that propagates to the left with speed $1.$ Moreover, the values of $u(t,x)$ in the r-wave depend linearly on $\frac{x}{t}.$ Let, as in the previous example, $Y(t,x,v)$ be the kinetic function of $u(x,t)$ and  $Y_\e(t,x,v){}={}Y(t,x,v)*\omega_\e(x),$
where for the definiteness we take $\omega_\e(x)$ to be  smooth, non-negative function, compactly supported on $[-2\e,2\e]$ and equal to $1$ on $[-\e,\e].$ It was shown in the previous example that for small times $Y_\e$ is a solution of \eqref{kinetic_4}. Consider time $t=1$ -- the moment the r-wave reaches shock. $Y_\e(1,x,v)$ is a smoothing in $x$ direction of the characteristic function of a triangle $\{(x,v)\,:\,x\in(0,2),\,v\in(x,1)\},$ and $Y_\e(0,x,v)$ is a smoothing in $x$ direction of the characteristic function of a square $\{(x,v)\,:\,x\in(0,1),\,v\in(0,1)\},$ 
For all small $\e,$ one directly verifies that
\[
\int_{-L}^L \int_0^1|Y^\e(1,x,v)|^2\,dxdv{}>{}\int_{-L}^L\int_0^1|Y^\e(0,x,v)|^2\,dxdv,
\]
violating the conservation property \eqref{kinetic_5:2}.

\end{section}

\end{document}